\documentclass[letter,10pt]{article}
\usepackage{amssymb}
\usepackage{amsthm}
\theoremstyle{dotless}
\usepackage{amsmath}
\usepackage{mathrsfs}
\usepackage{verbatim}
\usepackage{enumerate}

\newcommand{\R}{\mathbb{R}}
\usepackage{cite}
\usepackage{hyperref}
\usepackage{changes}
\usepackage{color}

\begin{document}
\title{Multi-marginal optimal transport: theory and applications\footnote{The author is pleased to acknowledge the support of a University of Alberta start-up grant and National Sciences and Engineering Research Council of Canada Discovery Grant number 412779-2012. 
}}

\author{Brendan Pass\footnote{Department of Mathematical and Statistical Sciences, 632 CAB, University of Alberta, Edmonton, Alberta, Canada, T6G 2G1 pass@ualberta.ca.}}
\maketitle
\begin{abstract} 
Over the past five years, multi-marginal optimal transport, a generalization of the well known optimal transport problem of Monge and Kantorovich, has begun to attract considerable attention, due in part to a wide variety of emerging applications.  Here, we survey this problem, addressing fundamental theoretical questions including the uniqueness and structure of solutions.  The answers to these questions uncover a surprising divergence from the classical two marginal setting, and reflect a delicate dependence on the cost function, which we then illustrate with a series of examples.    We go one to describe some applications of the multi-marginal optimal transport problem, focusing primarily on matching in economics and density functional theory in physics.

%This paper surveys known results on multi-marginal optimal transport.
\end{abstract}

\section{Introduction}
This paper surveys the theory and applications of multi-marginal optimal transport, the general problem of aligning or correlating several measures so as to maximize efficiency (with respect to a given cost function).  There are two precise formulations;  in the Monge formulation,  given compactly supported Borel probability measures $\mu_1,...,\mu_m$ (marginals) on smooth manifolds $M_1,...,M_m$, respectively, and a continuous cost function $c(x_1,...,x_m)$, ones seeks to minimize
\begin{equation}\label{mmm}
\int_{M_1 }c(x_1,F_2(x_1),...,F_m(x_1))d\mu_1(x_1)
\end{equation}
among all $(m-1)$-tuples of maps, $(F_2, F_3,...,F_m)$ such that $F_{i\#}\mu_1=\mu_i$  for $i=2,3,...m$.  Here, the \textit{push-forward} $F_{i\#}\mu_1$ of the measure $\mu_1$ by the map $F_i: M_1 \rightarrow M_i$ is the measure on $M_i$ defined by $(F_{i\#}\mu_1)(A) = \mu_1(F_i^{-1}(A))$ for all Borel $A \subset M_i$.

In the Kantorovich formulation, we seek to minimize 
\begin{equation}\label{mmk}
\int_{M_1 \times M_2 \times ... \times M_m}c(x_1,x_2,...,x_m)d\gamma(x_1,x_2,...,x_m)
\end{equation}
over the set $\Pi(\mu_1,\mu_2,...,\mu_m)$ of positive joint measures $\gamma$ on the product space $M_1 \times M_2 \times ... \times M_m$ whose marginals are the $\mu_i$.  Note that if $F_2,...,F_m$ satisfy the push-forward constraints in \eqref{mmm}, then $\gamma:=(Id,F_2,...,F_m)_{\#}\mu_1 \in \Pi(\mu_1,\mu_2,...,\mu_m)$, and 
$$
\int_{M_1 }c(x_1,F_2(x_1),...,F_m(x_1))d\mu_1(x_1) =\int_{M_1 \times M_2 \times ... \times M_m}c(x_1,x_2,...,x_m)d\gamma(x_1,x_2,...,x_m),
$$
and so problem \eqref{mmk} is a relaxation of \eqref{mmm}.

The Kantorovich problem \eqref{mmk} amounts to a linear minimization over a convex, weakly compact set; it is not difficult to assert existence of a solution.  Much of the attention in the literature has focused on uniqueness and the structure of the minimizer(s); in particular, a natural question is to determine when the solution concentrates on the graph of a function $(F_2,...,F_m)$ over the first marginal, in which case this function induces a solution to \eqref{mmm} (often called a \textit{Monge} solution).

When $m=2$, \eqref{mmm} and \eqref{mmk} correspond respectively to the classical (two marginal) optimal transport problems of Monge and Kantorovich, which have seen a great outpouring of results over the last 25 years, and remain very active research problems, at the interface of geometry, PDE, functional analysis and probability, with applications in physics, economics, fluid mechanics, meteorology, etc.  Theory in the two marginal setting is fairly well understood, and is exposed nicely in two books by Villani \cite{V,V2} and several survey papers, including those by Ambrosio-Gigli \cite{AG13}, Evans \cite{Evans} and McCann\cite{Mc14}.  In particular, it is well known that under mild conditions on the cost function and marginals, the solution to \eqref{mmk} is unique and is concentrated on the graph of a function, which in turn solves \eqref{mmm}.

The extension to $m \geq 3$ marginals is not as well understood, but has recently begun to attract a fair bit of attention, due to a diverse variety of emerging applications.  Our main goals in this manuscript are first to survey the theory of multi-marginal optimal transport problems, addressing questions on existence of Monge solutions, as well as uniqueness and structure of Kantorovich solutions, and secondly, to discuss the consequences and interpretation of this theory in the context of applications.  On the theoretical side, the multi-marginal literature is somewhat fractured, with many papers addressing particular cost functions, or providing generalizations of earlier results.  It is only recently that a clearer picture of the underlying structure has begun to emerge, and one of the aims of this paper is to present a more unified view of what is known about multi-marginal problems. 

We will attempt to frame the multi-marginal theory relative to the better understood two marginal theory; throughout the text, we will comment on how many of the multi-marginal results we develop compare to analogues from the two marginal setting and try to explain the differences. 
% as we develop multi-marginal theory, we will compare it to the analogous results in the two marginal setting  
In particular, we will formulate conditions that are analogous to the well known twist and non-degeneracy conditions, also known as \textbf{(A1)} and \textbf{(A2)}, introduced in the groundbreaking regularity theory of Ma, Trudinger and Wang \cite{mtw}; our conditions \textbf{(MMA1)} and \textbf{(MMA2)} ("multi-marginal" \textbf{(A1)} and \textbf{(A2)},  respectively)  imply structural results on optimal measures analogous to those implied by \textbf{(A1)} and \textbf{(A2)} in the two marginal setting.  

It will be clearly apparent that the conditions \textbf{(MMA1)} and \textbf{(MMA2)} are much more restrictive than their two marginal counterparts \textbf{(A1)} and \textbf{(A2)}, and this hints at a striking difference between two and multi-marginal problems.  A dichotomy has begun to emerge between multi-marginal costs which satisfy, for example, \textbf{(MMA1)}, in which case optimizers in \eqref{mmk} are concentrated on graphs over $x_1$ and are unique, and those which violate it, in which case solutions can concentrate on higher dimensional  submanifolds of the product space, and may be non unique.  This sensitivity to the cost function is largely absent from two marginal problems, and, after presenting the general theory, we illustrate it with a series of examples, exhibiting cost functions for which optimizers have Monge solutions and are unique, as well as some for which these properties fail.  Several of these examples are relevant in the applications discussed subsequently.

Among several  applications of multi-marginal optimal transport, we focus primarily on two, which reflect and illustrate the theory.  One comes from matching problems in economics, and here the theory mirrors the two marginal case fairly closely.  For the other, which originates in density functional theory (DFT), modeling electronic correlations in physics, the theory is quite different.    Our choice to focus here on matching problems and density functional theory stems from two facts: 1) the structure of optimal measures obtained in these applications is reasonably well understood and 2) these applications nicely demonstrate the strong qualitative dependence of the solution on the cost function.  The costs arising in matching problems satisfy (under certain weak hypotheses) \textbf{(MMA1)} and \textbf{(MMA2)}, resulting in low dimensional, unique solutions, closely resembling the two marginal theory, whereas the costs relevant to DFT permit non-unique, higher dimensional solutions.   For both of these applications, we describe in some detail the modeling process leading to the optimal transport problem, and then discuss what is known about the solution, leaning on the theory developed in the second section.  Other applications for multi-marginal problems or variants of them, in, for example, image processing and mathematical finance are also described briefly.

Let us emphasis that the question of whether or not solutions are of Monge type is vital from an applied point of view.  First, computationally, they represent a considerable dimensional reduction; for the same reason, even in the absence of Monge solutions, it is important to try and estimate the dimension of the sets on which the solution can concentrate, which is largely the focus of subsection \ref{dimsection}.  In addition, they can be interpreted as interesting phenomena in various applications; for example, Monge solutions are known as \textit{pure} solutions in matching problems, and \textit{strictly correlated electrons} in the DFT literature, concepts which we explain further in section 3  below.

%In the event that Monge solutions fail, it is still desirable to determine what

The paper is organized as follows.  Section 2 is devoted to theory; we first formulate the conditions \textbf{(MMA1)} and \textbf{(MMA2)}.  We then illustrate this theory with a series of examples, and close out the section with a brief discussion of variants and extensions of \eqref{mmk}. We turn our attention to applications in section 3; the discussion on matching and DFT problems is followed by a short overview of a variety of other applications.  Section 3 is then concluded with a brief description of numerical methods for \eqref{mmk}.

\section{Theory}
We will assume for simplicity that each $M_i$ shares a common dimension, $n:=dim(M_i)$, as this case is the most relevant in applications and has also received the most theoretical interest.   We note, however, that much of the theory, particularly in the first subsection, has partial extensions to problems where the dimensions differ.  Our main object of interest will be the \textit{support}, $spt(\gamma)$, of the optimal measure $\gamma$, which is defined as the smallest closed subset of $M_1 \times M_2 \times...\times M_m$ of full mass, $\gamma(spt(\gamma)) =1$.

We begin by recording some notation from differential geometry.  Let $T_{x_i}M_i$ denote the tangent space of $M_i$ at $x_i$, and $T_{x_i}^*M_i$ its dual, the cotangent space of $M_i$ at $x_i$.  Denote by $D_{x_i}c \in T_{x_i}^*M_i$  the differential of $c$ with respect to $x_i$, that is $D_{x_i}c=\frac{\partial c}{\partial x_i^{\alpha_i}}dx_i^{\alpha_i}$ in local coordinates.\footnote{Here and in what follows, summation on the repeated  index $\alpha_i$ is implicit, in accordance with the Einstein summation convention.}  For each $i \neq j$, we consider the bilinear form $D^2_{x_ix_j} c :T_{x_i}M_i \times T_{x_j}M_j \rightarrow \mathbb{R}$, defined in local coordinates by $D^2_{x_ix_j} c =\frac{\partial ^2 c}{ \partial x_j^{\alpha_j} \partial x^{\alpha_i}_i}dx_i^{\alpha_i} \otimes dx_j^{\alpha_j}$.\footnote{The same notation will sometimes be used to denote the obvious extension of this form to the whole tangent space, $T_{x_1}M_1 \times T_{x_2}M_2 \times...\times T_{x_m}M_m$. }   We will often identify  $D^2_{x_ix_j} c$ with the corresponding $n \times n$ matrix, which is nothing other than the matrix mixed second order partials $ D^2_{x_ix_j}c :=(\frac{\partial ^2 c}{\partial x^{\alpha_i}_i \partial x_j^{\beta_j}})_{\alpha_i, \alpha_j}$.

\begin{comment}\subsection{Existence of optimal measures and a dual problem}
We record here a classical result on the existence of minimizers to \eqref{mmk}.  As the theorem and proof are exactly analogous to well known results in the two marginal case, we only sketch the proof and refer to \cite{V} for details.

\newtheorem{Existence}{theorem}
\begin{Existence}
,
There exist a minimizer.  Furthermore, there exist measurable functions $u_1,u_2,...,u_m$ such that $\sum_{i=1}^mu_i(x_i) \leq  c(x_1,...,x_m)$, with equality $\gamma$ almost everywhere, for any solution $\gamma$.
\end{Existence}
\end{comment}
\subsection{Local structure of the optimizer: dimension of the support}\label{dimsection}
First, we discuss what can be said about the Hausdorff dimension of $spt(\gamma)$.  We first recall a result from the two marginal case, proved with McCann and Warren \cite{MPW}, which requires the following assumption, originally introduced in \cite{mtw}.

\newtheorem*{nondeg}{(A2)}
\begin{nondeg}
\textbf{(Non-degeneracy)} At a point $(x_1,x_2) \in M_1 \times M_2$, assume  the matrix $D_{x_1x_2}^2c(x_1,x_2)$ has full rank.  
\end{nondeg}
Under this condition, we have the following result from \cite{MPW}.
\newtheorem{rectifiability}{Theorem}[subsection]
\begin{rectifiability}\label{rectifiability}\textbf{(Local $n$-rectifiability for two marginal problems)}
Let $m=2$.  Assume $c$ is non-degenerate at some point $(x_1,x_2)$.  Then there is a neighbourhood $N$ of $(x_1,x_2)$ such that, for any optimal measure, $\gamma$, $N \cap spt(\gamma)$ is contained in an $n$-dimensional Lipschitz submanifold of the product space.
\end{rectifiability}

An analogous result governing the behaviour of multi-marginal optimizers was proven by the author in \cite{P} and is given below.  It helps classify the allowable dimensions of the support of $\gamma$.  The statement of the result requires some additional notation.

We let $P$ be the set of partitions of $\{1,2,....,m\}$ into two non empty  disjoint subsets; that is, $p:=\{p_+,p_-\} \in P$ means that $p_+ \cup p_- = \{1,2,....,m\}$, $p_+ \cap p_- =\emptyset$ and $p_+,p_- \neq \emptyset$.  For each $p:=\{p_+,p_-\} \in P$, we define the bilinear form $g_p$ on $M_1 \times M_2 \times...\times M_m$ by $g_p = \sum_{i\in p_+,j\in p_-}(D^2_{x_ix_j}c +D^2_{x_jx_i}c)$.  We can identify $g_p$ with an $mn \times mn$ matrix whose  $i,j$ block is zero if $i$ and $j$ both belong to either $p_+$ or $p_-$ and $D^2_{x_ix_j} c$ otherwise. Define $G:=\{\sum_{p \in P}t_p g_p: t_p \geq 0, \sum_{p \in P}t_p=1\}$ to be the convex hull generated by the $g_p$.

Note that the diagonal blocks of any $g \in G$ are $n \times n$ $0$ matrices.  The off-diagonal blocks are nonnegative multiples 
\begin{equation*}
\sum_{\substack{p=\{p_+,p-\}\\  x_i \in p_+\text{ and }x_j \in p_- \textit{ or } \\x_j \in p_+\text{ and }x_i \in p_-}}t_p  D^2_{x_ix_j}c
\end{equation*}
 of the $D^2_{x_ix_j} c$.
Note that, as $ D^2_{x_ix_j}c= \big(D^2_{x_jx_i}c\big)^T $, each $g \in G$ is symmetric and therefore its signature (the number of positive, negative, and zero eigenvalues, respectively) is well defined.  The following theorem, proved in \cite{P},  controls the dimension of the support of the optimizer(s) $\gamma$ in terms of these signatures.

\newtheorem{spacelike}[rectifiability]{Theorem}
\begin{spacelike}\label{spacelike}\textbf{(Dimensional bounds on $spt(\gamma)$ for multi-marginal problems)}
Suppose that at some point $x \in M_1 \times M_2 \times...\times M_m$, the signature of some $g \in G$ is $(\lambda_+,\lambda_- , mn-\lambda_+-\lambda_-)$.  Then there exists a neighbourhood $N$ of $x$ such that $N \cap spt(\gamma)$ is contained in a Lipschitz submanifold of the product space with dimension no greater than  $mn-\lambda_+$.  If $spt(\gamma)$ is differentiable at $x$, it is \emph{timelike} for $g$; that is $v^tgv \leq 0$ for any $v \in T_x spt(\gamma)$ in the tangent space of $spt(\gamma)$.
\end{spacelike}

\newtheorem{bestsig}[rectifiability]{Remark}
\begin{bestsig}
Because the zero diagonal blocks of each $g \in G$ ensure the existence of an $n$-dimensional \emph{lightlike} subspace of $T_x(M_1 \times M_2 \times...\times M_m)$ (that is, a subsapce $S$ for which $v^tgv =0$ for all $v \in S$), standard linear algebra arguments imply that we must have $\lambda_- \leq (m-1)n$; similarly, $\lambda_+ \leq (m-1)n$.  Therefore, the smallest bound on the dimension of $spt(\gamma)$ which Theorem \ref{spacelike} can provide is $n$.   Provided that $g$ is invertible, the bound will be at most $(m-1)n$, the largest allowable value of $\lambda_-=mn-\lambda_+$.
\end{bestsig}
With Theorem \ref{rectifiability} in mind, we introduce the following analogue of $\textbf{(A2)}$.
\newtheorem*{mmnondeg}{(MMA2)}
\begin{mmnondeg}
At a point $x \in M_1 \times M_2 \times...\times M_m$, assume that there is some $g \in G$ with signature $((m-1)n,n,0)$.
\end{mmnondeg}
As an immediate application of we have the following analogue of Theorem \ref{rectifiability}, which serves as the impetus behind the nomenclature \textbf{(MMA2)}.

\newtheorem{ndimsol}[rectifiability]{Corollary}

\begin{ndimsol}
If $c$ satisfies \textbf{(MMA2)} at some point $x$, then there is a neighbourhood $N$ of $x$ such that $N \cap spt(\gamma)$ is contained in an $n$-dimensional Lipschitz submanifold of the product space.
\end{ndimsol}

\newtheorem{comp2marg}[rectifiability]{Remark}
\begin{comp2marg}
When $m=2$, the only $g \in G$, up to a positive multiplicative constant, is 
\begin{equation*}\bar g:=
\begin{bmatrix}
0 & D^2_{x_1x_2}c & \\
D^2_{x_2x_1}c  & 0& \\
\end{bmatrix}.
\end{equation*}
This is exactly the pseudo-metric introduced by Kim-McCann in the study of the covariant theory of the regularity of optimal maps \cite{KM}.  They noted that this $g$ has signature $(n,n,0)$ whenever $c$ is non-degenerate; therefore Theorem \ref{spacelike} generalizes Theorem \ref{rectifiability}.  In fact, Theorem \ref{spacelike} applies even when non-degeneracy fails, and so provides new information even in the two marginal case; in this case, the signature of $g$ is $(r,r,2n-2r)$, where $r$ is the rank of  $ D^2_{x_1x_2}c$ \cite{P}.
\end{comp2marg}

For $m \geq 3$, there is of course a lot of choice in the way we choose the $t_p$'s, and for a particular cost one can optimize the choice to get the best bound in Theorem \ref{spacelike}.  In our applications, we will mostly focus on the simplest case when all the $t_p$'s are taken to be the same, in which case we obtain (up to a multiplicative constant) the off diagonal part of the matrix of second derivatives of $c$, given in block form by
\begin{equation*}\bar g:=
 \begin{bmatrix}
0 & D^2_{x_1x_2}c & D^2_{x_1x_2}c&...& D^2_{x_1x_m}c\\
D^2_{x_2x_1}c  & 0&D^2_{x_2x_3}c&...&D^2_{x_2x_m}c \\
D^2_{x_3x_1}c  & D^2_{x_3x_2}c& 0&...&D^2_{x_3x_m}c\\
...  & ...&...&....&... \\
D^2_{x_mx_1}c  & D^2_{x_mx_2}c  &...&...&0 \\
\end{bmatrix}.
\end{equation*}
However, other choices can be useful; for example, one can show that if $m$ is even, for a generic cost the dimension of the support is no more than $\frac{mn}{2}$ \cite{P}.
%Note that at this point, we begin to see the theory of multi-marginal problems diverge from the classical two marginal theory.

Finally, let us note that condition \textbf{(MMA2)} is much more restrictive than \textbf{(A2)}, which will, roughly speaking, be satisfied by a generic cost $c$ at a generic point $(x_1,x_2)$.    On the other hand, \textbf{(MMA2)} implies negative definiteness of the symmetric part of $D^2_{x_ix_j}c[D^2_{x_kx_j}c]^{-1}D^2_{x_kx_i}c$ for all distinct $i,j,k$, which is certainly not generically true \cite{P4}.  In subsection 2.3, we will see several examples where \textbf{(MMA2)} fails, and the solution concentrates on sets with dimension larger than $n$.
\subsection{Uniqueness and graphical structure of optimal measures}
We now turn to the question of when the optimizer has  Monge, or graphical, structure.   For two marginal problems, the \textit{twist} condition suffices for this, and also implies uniqueness of the optimal measure.
\newtheorem*{twist}{(A1)}
\begin{twist}
\textbf{(Twist)} Assume that $c$ is semi-concave and that for each fixed $x_1$, the map
\begin{equation*}
x_2 \mapsto D_{x_1}c(x_1,x_2)
\end{equation*}
is injective on the subset $\{x_2: D_{x_1}c(x_1,x_2) \text{ exists }\} \subseteq M_2$ where $c$ is differentiable with respect to $x_1$.
\end{twist}

The following result is well known; versions of it can be found, in comparable generality, in Caffarelli \cite{Caf}, Gangbo \cite{G}, Gangbo and McCann \cite{GM} and Levin \cite{lev}.
\newtheorem{mapping}{Theorem}[subsection]
\begin{mapping}\label{mapping}
\textbf{(Monge solutions and uniqueness for two marginal problems)} Suppose $m=2$.  Assume that the first marginal $\mu_1$ is absolutely continuous with respect to local coordinates and that $c$ satisfies \textbf{(A1)}.  Then the optimal measure $\gamma$ is concentrated on the graph of a function $F: M_1 \rightarrow M_2$.  This mapping is a solution to Monge's problem, and the solutions to both Monge's problem and Kantorovich's are unique.
\end{mapping}

When $m \geq 3$, the most general known condition for Monge solutions and uniqueness was formulated with Kim \cite{KP2} and requires  the following notion:
$S \subset M_2 \times M_3 \times ...\times M_m$ is a \textit{splitting set} at $x_1 \in M_1$ if there exist functions $u_i :M_i \rightarrow \mathbb{R}$, for $i=2,3,...m$, such that

\begin{equation*}
\sum_{i=2}^m u_i(x_i) \leq c(x_1,...,x_m) 
\end{equation*}
with equality on $S$.
Splitting sets arise in optimal transport problems in connection with the duality theorem (see \cite{K} for a multi-marginal version); essentially, this shows that the set of all points which are optimally coupled to $x_1$ is a splitting set at $x_1$.  

Our multi-marginal version of the twist condition is the following.
\newtheorem*{twistss}{(MMA1)}
\begin{twistss}
\textbf{(Twist on splitting sets)}
Assume that $c$ is semi-concave and, whenever $S \subset M_2 \times M_3 \times ...\times M_m$ is a splitting set at a fixed $x_1 \in M_1$, the map

\begin{equation}
(x_2,...,x_m) \mapsto D_{x_1}c(x_1,x_2,...,x_m)
\end{equation}
is injective on the subset of $S$ on which $ D_{x_1}c$ exists, $S \cap Dom( D_{x_1}c)$.
\end{twistss}

Under this condition, we have the following analogue of Theorem \ref{mapping}, proved with Kim \cite{KP}.

\newtheorem{mmmapping}[mapping]{Theorem}
\begin{mmmapping}\label{mmmapping}
\textbf{(Monge solutions and uniqueness for multi-marginal problems)} Assume that $\mu_1$ is absolutely continuous with respect to local coordinates and that $c$ satisfies \textbf{(MMA1)}.  Then the optimal measure $\gamma$ is concentrated on the graph of a function $F =(F_2,F_3,...,F_m): M_1 \rightarrow M_2 \times M_3 \times ...\times M_m$.  This mapping is a solution to Monge's problem, and the solutions to both Monge's problem and Kantorovich's are unique.
\end{mmmapping}
%The proof can be found in \cite{KP} essentially shows that for $\gamma$ 
Let us mention as well the recent paper of Moameni, which shows that under a local version of  \textbf{(MMA1)}, one can show that $spt(\gamma)$ is concentrated on  the union of several graphs \cite{moameni2}.

The twist on splitting sets is complicated and difficult to verify directly.  There are, however, a number of known classes of examples which satisfy it (many of these are described in the following subsection), as well as sufficient (but not necessary) local differential conditions on the cost, which we formulate now.  %We describe several examples below.

\subsubsection{Local differential conditions for twistedness  on splitting sets}
Let $M_i \subseteq \R^n$ for each $i=1, ..., m$.\footnote{The conditions in this section can actually be developed on somewhat more general spaces (essentially subsets on $\mathbb{R}^n$, but with non-Euclidean metrics); see \cite{P1}.} The most restrictive part of the differential conditions is based on the following tensor.

%Assume  that $c$ is $(1,m)$-twisted; that is, the map

%\begin{equation*}
%x_m \mapsto D_{x_1}c(x_1,x_2,...x_m)
%\end{equation*}
%is injective for fixed $x_1,...x_{m-1}$. We will also assume that $c$ is $(1,m)$-non-degenerate; that is, the matrix $D^2_{x_1x_m}c $.

\newtheorem{tensor}[mapping]{Definition}

\begin{tensor}
Suppose $c$ is $(1,m)$-non-degenerate (that is, the matrix $D^2_{x_1x_m}c $ is non-singular.) Let $y=(y_1,y_2,...,y_m) \in M_1 \times M_2 \times...\times M_m$.  For each $i:=2,3,...,m-1$ choose a point $y(i)=(y_1(i),y_2(i),...,y_m(i)) \in \overline{M_1} \times \overline{M_2} \times...\times \overline{M_m}$ such that $y_i(i)=y_i$.  Define the following bi-linear maps on $T_{y_2}M_2 \times T_{y_3}M_3 \times ...\times T_{y_{m-1}}M_{m-1}$:

\begin{equation*}
 S_{y}=-\sum_{j=2}^{m-1} \sum_{\substack {i=2 \\ i \neq j}}^{m-1}D^2_{x_ix_j}c(y) +\sum_{i,j=2}^{m-1}(D^2_{x_ix_m} c \, (D^{2}_{x_1x_m}c)^{-1}\, D^2_{x_1x_j}c)(y)
\end{equation*}

\begin{equation*}
 H_{y,y(2),y(3),...,y(m-1)}=\sum_{i=2}^{m-1}( D^2_{x_i x_i} c(y (i))- D^2_{x_ix_i}c(y))
\end{equation*}

\begin{eqnarray*}
T_{y,y(2),y(3),...,y(m-1)}=S_{y}+H_{y,y(2),y(3),...,y(m-1)}
\end{eqnarray*}
\end{tensor}
The main condition we will impose is negative definiteness of the tensor $T$, for all choices of the $y,y(2),y(3),...,y(m-1)$; a structural condition on the domains, defined in terms of the following set, is needed as well.

\newtheorem{domains}[mapping]{Definition}
\begin{domains}\label{domains}
 Let $x_1 \in M_1$ and $p_1 \in T^*_{x_1}M_1$.  We define $Y^c_{x_1,p_1} \subseteq M_2 \times M_3 \times ...\times M_{m-1}$ by
\begin{equation*}
 Y^c_{x_1,p_1} =\{(x_2,x_3,...,x_{m-1}) |\text{ } \exists\text{ } x_m \in M_m \textit{ s.t. }  D_{x_1}c(x_1,x_2,...,x_m) = p_1\}
\end{equation*}

\end{domains}

These conditions were introduced in \cite{P1}, where it was proven that they imply Monge solution and uniqueness results for the multi-marginal problem, before the twist on splitting set condition had been formulated.  The following result asserts that they indeed imply the twist on splitting sets condition; a proof can be found in \cite{KP2}.

\newtheorem{diffcond}[mapping]{Proposition}
\begin{diffcond}\label{diffcond}
(Sufficient differential conditions) 

Suppose that:
\begin{enumerate}
\item $c$ is $(1,m)$-non-degenerate; that is,  $D^2_{x_1x_m}c $ is non-singular everywhere.
\item $c$ is $(1,m)$-twisted; that is, $x_m \mapsto D_{x_1}c(x_1,x_2,...x_m)$  is injective for each fixed fixed $x_1,...x_{m-1}$.
\item For all choices of $y=(y_1,y_2,...,y_m) \in M_1 \times M_2 \times...\times M_m$ and of $y(i)=(y_1(i),y_2(i),...,y_m(i)) \in \overline{M_1}\times \overline{M_2} \times...\times \overline{M_m}$ such that $y_i(i)=y_i$ for $i=2,...,m-1$, we have 

\begin{equation*}
T_{y,y(2),y(3),...,y(m-1)} < 0. \footnote{Note that the matrix $T_{y,y(2),y(3),...,y(m-1)}$ need not be symmetric; the condition $T_{y,y(2),y(3),...,y(m-1)} < 0$ means that $V^T\cdot T_{y,y(2),y(3),...,y(m-1)} \cdot V< 0$ for all $V$, or, equivalently, that the symmetric part of $T_{y,y(2),y(3),...,y(m-1)}$ is negative definite.}
\end{equation*}
\item For all $x_1 \in M_1$ and $p_1 \in T^*_{x_1}M_1$, $Y_{x_1,p_1}^c$ is  convex.
\end{enumerate}
Then $c$ is twisted on splitting sets.
\end{diffcond}

Let us remark that an example in \cite{P13a} demonstrates that these conditions are not necessary for the twist on splitting set condition.  Some examples of cost functions satisfying these conditions can be found in sections 2.3.2 and 2.3.5 below; additional examples can be found in \cite{P4}.
%\newtheorem{oned}{Example}
%\begin{oned}
%A fair bit of intuition on multi-marginal optimal transport can be gleaned from the one dimensional case.
%\end{oned}

\subsection{Examples}
Here we illustrate the theory from the previous few subsections by studying several examples.
\subsubsection{Marginals with one dimensional support}
Naturally, the simplest multi-marginal problems occur when the marginals are supported on intervals $M_i=(a_i,b_i)  \subset \mathbb{R}$.  It it well worth studying the one dimensional case, as a lot of intuition can be gleaned from it which also applies to higher dimensional problems.

In this setting, the non-degeneracy condition \textbf{(A2)} amounts to the condition $c_{x_1x_2}:=\frac{\partial ^2 c}{\partial x_1 \partial x_2} \neq 0$, which means either $c_{x_1x_2} > 0$ or $c_{x_1x_2} < 0$ (this condition also implies \textbf{(A1)}).  In this case, it is well known that the one dimensional set in Theorem \ref{rectifiability} above is either monotone increasing (if  $c_{x_1x_2} < 0$) or monotone decreasing  (if  $c_{x_1x_2} >0$).  

For multi-marginal problems, the interaction between the various mixed second order partials becomes relevant.  We say that the cost $c$ is \textit{compatible} if  
\begin{equation*}
c_{x_ix_j} (c_{x_k x_j} )^{-1}c_{x_kx_i} <0
\end{equation*}
everywhere, for all distinct $i,j,k$.  Monge solution and uniqueness results for compatible costs were essentially established in\cite{C}\footnote{This paper actually focused on submodular costs, which essentially means $ c_{x_ix_j} <0$ for all distinct $i,j$; compatibility is equivalent to submodularity, up to a change of variables; see \cite{P4}.}, while an alternate argument, similar in spirit to Theorem \ref{spacelike} was developed in \cite{P4}; in fact, a rearrangement inequality that is essentially equivalent can be traced back to Lorenz \cite{lorentz}.  The arguments in \cite{C} and \cite{P4} essentially showed that a splitting set $S$ must be $c$-monotone, in the sense that if both $(x_2,...,x_m), (y_2,...,y_m) \in S$, then $(x_i-y_i)c_{x_ix_j} (x_j-y_j) \leq 0$ for all $i \neq j$.  We can then show that compatibility implies \textbf{(MMA1)}, placing these results within the framework of the previous subsection:
\newtheorem{onedmonge}{Theorem}[subsubsection]
\begin{onedmonge}\label{onedmonge}
Suppose $c$ is compatible.  Then it satisfies \textbf{(MMA1)}.
\end{onedmonge}
\begin{proof}
Suppose $S \subseteq M_2 \times...\times M_m$ is a splitting set at $x_1$  and $x=(x_2,...,x_m),y= (y_2,...,y_m) \in S$ with 
\begin{equation}\label{firstorder}
\frac{\partial c}{\partial x_1}(x_1,x_2,...,x_m) =\frac{\partial c}{\partial x_1}(x_1,y_2,...y_m) .
\end{equation}
We need to show $x=y$.  Letting 
\begin{equation*}
x(s) =(x_1, sx_2 + (1-s)y_2,...,sx_m + (1-s)y_m),
\end{equation*}
for  $ s \in [0,1]$, the equality \eqref{firstorder} can be written as
\begin{equation*}
\int_0^1\sum_{i=2}^mc_{x_1x_i}(x(s))(x_i-y_i)ds =0
\end{equation*}
Now, multiplying this by $(x_2-y_2)c_{x_1x_2}(x(0))$, we have
\begin{eqnarray*}
&&(x_2-y_2)^2c_{x_1x_2} (x(0))\int_0^1c_{x_1x_2} (x(s))ds\\
&&\hspace{0.75 in}+(x_2-y_2)c_{x_1x_2} (x(0))\int_0^1\sum_{i=3}^mc_{x_1x_i} (x(s))(x_i-y_i)ds =0\\
\end{eqnarray*}
or
\begin{eqnarray*}
&&(x_2-y_2)^2c_{x_1x_2}(x(0))\int_0^1c_{x_1x_2}(x(s))ds\\
&&\hspace{0.75 in}+\int_0^1\sum_{i=3}^m(x_2-y_2)c_{x_1x_2} (x(0))    \frac{  c_{x_ix_2} (x)    }{c_{x_ix_2}(x) }    c_{x_1x_i} (x(s))(x_i-y_i)ds =0\\
\end{eqnarray*}
Now note that the first term on the left hand side above is clearly non-negative, as $c_{x_1x_2}$ cannot change sign by the compatibility condition (a change of sign would imply  a zero of $c_{x_1x_2}$, which would violate the strict inequality in the compatibility condition).  On the other hand, as $   \frac{  c_{x_1x_2} (x)  }{c_{x_ix_2}(x) }    c_{x_1x_i} (x(s)) <0$ by compatibility and   $(x_2-y_2)c_{x_ix_2}(x)  (x_i-y_i)\leq 0$ by the splitting set property, the remaining terms are nonnegative as well.  Thus, the only possibility is that all terms are zero, in particular, considering the first term, this implies $x_2=y_2$. A similar argument implies $x_i=y_i$ for $i=3,...m$, which implies the twist on splitting set property.
\end{proof}
\newtheorem{comprem}[onedmonge]{Remark }
\begin{comprem}
(\textbf{Heuristic interpretation of compatibility}) Looking at the signs of $c_{x_ix_j}$, we can guess, based on what is known about the $m=2$ case whether the correlation between $x_i$ and $x_j$ should be positive (if $c_{x_ix_j}<0$, corresponding to an increasing graph) or negative (if $c_{x_ix_j}>0$ corresponding to a decreasing graph).  The compatibility condition ensures that these guesses are consistent; from this perspective it is not surprising that in this case we get Monge type solutions (with correlations as predicted by the signs of the $c_{x_ix_j}$).  On the other hand, when compatibility fails, the competing interactions can cause the solution to concentrate on higher dimensional sets.  For example, the cost $c(x_1,x_2,x_3) =(x_1+x_2+x_3)^2$ is not compatible; looking at the pairwise interactions $ c_{x_ix_j}$ might lead us to expect pairwise monotone decreasing relationships of the variables.  However, as  monotone decreasing dependences of $(x_1,x_2)$ and $(x_1,x_3)$ together imply a  monotone \emph{increasing} dependence of $(x_2,x_3)$, this is not consistent.  This cost allows optimizers with $2$-d support; as the cost is $0$ on the plane $x_1+x_2+x_3=0$ and positive elsewhere, it is clear that any measure concentrated on this plane is optimal for its marginals.
\end{comprem}
When $m=3$, an easy calculation shows that compatibility is equivalent to \textbf{(MMA2}); for larger $m$ it is necessary but I do not know if it is sufficient.  It is interesting that these last two facts have very natural generalizations to higher dimensions; for $m=3$, \textbf{(MMA2}) is equivalent to $D^2_{x_2x_1}c[D^2_{x_3x_1}c]^{-1}D^2_{x_3x_2}c<0$ while for larger $m$ the condition $D^2_{x_ix_j}c[D^2_{x_kx_j}c]^{-1}D^2_{x_kx_i}c<0$ for all distinct $i,j,k$ is necessary (but possibly not sufficient) for \textbf{(MMA2}) \cite{P4}.

\subsubsection{Functions on the sum}
Suppose each $M_i \subseteq \mathbb{R}^n$ and $c(x_1,x_2,...,x_m) = h(\sum_{i=1}^m x_i)$ for some $C^2$ smooth $h: \mathbb{R}^n \rightarrow \mathbb{R}$ .   Then each $D^2_{x_ix_j}c =D^2h$, so 
\begin{equation*}\bar g=
 \begin{bmatrix}
0 & D^2h & D^2h&...& D^2h\\
D^2h & 0&D^2h&...&D^2h \\
D^2h  & & 0&...&D^2_h\\
...  & ...&...&....&... \\
D^2h  & D^2h  &...&...&0 \\
\end{bmatrix}
\end{equation*}
It is then an easy exercise to determine the signature of $\bar g$ in terms of the signature of $D^2h$ (see \cite{P}\cite{P4} for the calculation and general formula).  When  $h$ is uniformly concave, $D^2h<0$, the signature of $\bar g$ turns out to be $((m-1)n,n,0)$ and so $c$ satisfies \textbf{(MMA2)}.   Furthermore, $c$ also satisfies \textbf{(MMA1)}; one can show this either by computing the local differential conditions  in Proposition \ref{diffcond} (see the calculation in \cite{P4}) or by noticing that $h(\sum_{i=1}^m x_i)=\inf_z \sum_{i=1}^m x_i \cdot z -h^*(z)$ is of the infimal convolution form discussed below in subsection \ref{infcon}, and so \textbf{(MMA1)} follows from Proposition \ref{infcosttwist}.  

%Note that this class of cost functions was considered by

Historically, concave functions of the sum were among the first multi-marginal cost functions to be studied.  Partial results for the special case $c=-|(\sum_{i=1}^m x_i)|^2$, or equivalently $\sum_{i=1}^m|x_i-x_j|^2$ can be traced back to Olkin and Rachev \cite{OR}, Knott and Smith \cite{KS} and Ruschendorf and Uckelmann \cite{RU}.  Full Monge solution and uniqueness results for this cost were then proven by Gangbo and Swiech \cite{GS}, while the extension to general strictly concave $h$ is due to Heinich \cite{H}.  All of this took place long before the general condition \textbf{(MMA1)} had been developed, but, from our perspective, one can view the Gangbo-Swiech proof as verifying the twist on splitting sets condition for this cost, and also serving as a pioneering, general prototype this type of argument.
%Monge solution and uniqueness results for this cost were obtained by Heinich \cite{H}.  Before that Gangbo and Swiech considered the special
\newtheorem{highdconvex}{Remark}[subsubsection]
\begin{highdconvex}\label{highdconvex}
\textbf{(High dimensional solutions for convex costs)} On the other hand, if $h$ is uniformly convex, so $D^2h>0$, then the twist on splitting sets condition fails.  In addition, the calculation in \cite{P4} shows that $\bar g$ has signature  $(n,(m-1)n,0)$, and so Theorem \ref{spacelike}  guarantees only that the solution is concentrated on a set of dimension at most $(m-1)n$.    In fact, this estimate is sharp; as is shown in \cite{P}, any measure concentrated on the set $\{\sum_{i=1}^mx_i=a\}$, for some constant $a \in \mathbb{R}^n$ is optimal for its marginals; furthermore, as is shown in \cite{P4} in these high dimensional cases, the solution may also be non unique, as on a higher dimensional surface there can be enough wiggle room to construct more than one measure with common marginals. This  represents a significant divergence from the two marginal theory, where uniqueness and $n$-dimensional solutions are much more generic; note, for example, that when $m=2$, the cost $h(x_1 +x_2)$ for uniformly convex $h$ satisfies both \textbf{(A1)} and \textbf{(A2)}.
\end{highdconvex}
\subsubsection{Radially symmetric problems}
An interesting class of examples is those for which $M_i \subset \mathbb{R}^n$ for each $i$ and the cost function is radially symmetric; that is, $c(Ax_1,Ax_2,....,Ax_m) =c(x_1,x_2,...x_m)$ for all rotations $A \in SO(n)$.  This class includes the Gangbo-Swiech cost $\sum_{i=1}^m|x_i-x_j|^2$ \cite{GS}, the determinant cost of Carlier and Nazaret \cite{CN}, $-det(x_1x_2...x_m)$ (when $m=n$, so that the matrix $(x_1x_2...x_m)$ is square) and the Coulomb cost $\sum_{i \neq j}\frac{1}{|x_i-x_j|}$ \cite{CFK}\cite{bdpgg}.

A measure $\mu$ is called radially symmetric if $A_{\#} \mu =\mu$ for all $A \in SO(n)$.  The theorem below was first proven for the determinant cost function in\cite{CN} and a proof of the general case can be found in \cite{P13b}; both of these arguments yield explicit constructions of the optimal $\gamma$.  Interesting, similar results can be proven using ergodic theory \cite{moameni}.

\newtheorem{radial}{Theorem}[subsubsection]
\begin{radial}\label{radial}
Assume $c$ and each $\mu_i$ is radially symmetric.  Then there is an optimal measure $\gamma$ in \eqref{mmk} such that:
\begin{enumerate}
\item $\gamma$ is radially symmetric; that is, $(A,A,...,A)_{\#}\gamma =\gamma$ for all $A \in SO(n)$.
\item For each $(x_1,x_2,...x_m) \in spt(\gamma)$, we have 
\begin{equation}\label{spheremin}
(x_1,x_2,...x_m) \in argmin_{|y_i|=r_i; i=1,2,...m}c(y_1,y_2,...,y_m),
\end{equation}
where $r_i =|x_i|$.
\end{enumerate}
\end{radial}

The reader will probably not find this surprising, and the first part  is in fact an immediate corollary of the uniqueness in Theorem \ref{mmmapping}, when it applies (for example, for the Gangbo-Swiech cost). What is perhaps more interesting is that for other costs, for which that theorem fails, we can use the construction to exhibit examples where the solution concentrates on a higher dimensional set and fails to be unique.

We will call a cost \textit{non-attractive} if for any radii $(r_1,...r_m)$ the minimizers $(x_1,x_2,...x_m) \in argmin_{|y_i|=r_i; i=1,2,...m}c(y_1,y_2,...,y_m)$ are not all co-linear; that is, $x_i \neq \pm \frac{r_i}{r_1}x_1$ for some $i$. Note that this condition is not satisfied for the Gangbo-Swiech cost, but is satisfied for the Coulomb and determinant costs.  Under this condition, the solution in Theorem \ref{radial} above is  not of Monge type for $n \geq 3$.

\newtheorem{highdnonattractive}[radial]{Corollary}
\begin{highdnonattractive}\label{highdnonattractive}
Suppose $c$ is radially symmetric and non-attractive and the marginals $\mu_i$ are radially symmetric and absolutely continuous with respect to Lebesgue measure.  Then there exists solutions $\gamma$ whose support is at least $2n-2$-dimensional.
\end{highdnonattractive}
\begin{proof}
The proof is given in \cite{P13b}; we sketch it here to bring out the main idea.
By Theorem \ref{radial} and the non-attractive condition, we can find $(x_1,...,x_m) \in spt(\gamma)$ such that $x_1$ and $x_i$ are not co-linear for some $i$.  Now note that there is an entire family of rotations $A$ that fix $x_1$ but not $x_i$ and the points $(Ax_1,Ax_2...Ax_m) =(x_1,Ax_2...Ax_m) $ are in the support of $\gamma$ for each such $A$.  Indeed, this family has dimension $n-2$, and so the support of $\gamma$ has dimension at least $n+n-2=2n-2$ (as we can choose freely the $n$-dimensional variable $x_1$ and the $n-2$ dimensional variable $A$).
\end{proof}
Note that his corollary, combined with Theorems \ref{mmmapping} and \ref{spacelike} implicitly yields that non-attractive, radially symmetric costs can satisfy neither \textbf{(MMA1)} nor \textbf{(MMA2)} when $n \geq 3$.
Finally, we remark that one can in fact often show that these higher dimensional solutions are non unique as well; see, for example \cite{CN} and \cite{P13b}.
\subsubsection{Infimal convolution costs}\label{infcon}
Here we focus on costs of the form 
\begin{equation}\label{infconcost}
c(x_1,x_2,...,x_m) =\min_{y \in Y}\sum_{i=1}^mc_i(x_i,y),
\end{equation}
where $Y$ is an additional smooth $n$-dimensional manifold.  As we will see in the next section, these types of costs arise naturally in applications in economics.  Costs of this form generally tend to be quite well behaved; the following result from \cite{P} yields quite generic conditions under which $c$ satisfies \textbf{(MMA2)}.

\newtheorem{infcostnondeg}[radial]{Proposition}
\begin{infcostnondeg}
Assume:
\begin{enumerate}
\item For all $i$,  $c_i$ is $C^{2}$ and non-degenerate; that is, $D^{2}_{x_iy}c_i$  is everywhere non-singular.
\item For each $(x_1,x_2,...,x_m)$ the minimum is attained by a unique $\tilde y(x_1,x_2,...,x_m) \in Y$ and
\item $\sum_{i=1}^m D^2_{yy}c_i(x_i,\tilde y(x_1,x_2,...,x_m))$ is non-singular.  
\end{enumerate}
Then the signature of $\bar{g}$ is $((m-1)n,n,0)$, and so $c$ satisfies \textbf{(MMA2)}.
\end{infcostnondeg}
We now turn to the twist on splitting sets condition.  Let us note that Monge solution results for costs of this form were proved (under successively weaker conditions on the $c_i$) in \cite{P1} \cite{P13a} and \cite{KP2}; the following is a special case of an example in \cite{KP2}, where costs with a more general infimal convolution form were considered.

\newtheorem{infcosttwist}[radial]{Proposition}
\begin{infcosttwist}\label{infcosttwist}
Assume $c_1$ is $(x_1,y)$-twisted  (ie, $y \mapsto D_{x_1}c_1(x_1,y)$ is injective), and $c_i$ is $(y,x_i)$-twisted  (ie, $x_i \mapsto D_{y}c_i(x_i,y)$ is injective) for $i=2,...,m$.  Then $c$ given by \eqref{infconcost} satisfies \textbf{(MMA1)}.
\end{infcosttwist}
\begin{proof}
A more general result is proved in \cite{KP2}.  Here, we feel it is instructive to outline the proof in this special case.

Fix $x_1$ and a splitting set $S$ at $x_1$.  We need to show that if $D_{x_1}c(x_1,...,x_m) = D_{x_1}c(x_1,\bar x_2,...\bar x_m)$, for $(x_2,...,x_m), (\bar x_2,...,\bar x_m) \in S$ then $(x_2,...,x_m)= (\bar x_2,...,\bar x_m) $.  We will show that $x_2=\bar x_2$; the argument that $x_j=\bar x_j$ for $j \neq 2$ is identical.

Choose $\tilde y \in argmin_y \sum_{i=1}^m c_i(x_i,y)$  attaining the minimum in  \eqref{infconcost}.  The semi-concave function, $ y \mapsto \sum_{i=1}^m c_i(x_i,y)$ is differentiable at its minimum $\tilde y$, and 
\begin{equation}\label{firstordcond}
\sum_{i=1}^m D_yc_i(x_i,\tilde y)=0.
\end{equation}
By semi-concavity, the existence of $D_{x_1}c(x_1,x_2,...,x_m)$ implies the existence of $D_{x_1}c_1(x_1,\tilde y)$, (as $c_1(x_1,\tilde y)$ is a supporting function for $c(x_1,x_2,...,x_m)$), and we have the equality 

\begin{equation*}
D_{x_1}c(x_1,x_2,...,x_m)=D_{x_1}c(x_1,\tilde y).
\end{equation*}
Similarly, for $\bar{\tilde y} \in argmin_y[ \sum_{i=2}^m c_i(\bar x_i,y) +c_1(x_1,y)]$, we have

\begin{equation*}
D_{x_1}c(x_1,\bar x_2,...,\bar x_m)=D_{x_1}c(x_1,\bar{\tilde y}).
\end{equation*}

But then, by our assumption  $D_{x_1}c(x_1,...,x_m) = D_{x_1}c(x_1,\bar x_2,...\bar x_m)$, we have $D_{x_1}c(x_1,\tilde y)=D_{x_1}c(x_1,\bar{\tilde y}) $.  $(x_1,y)$ -twistedness then implies 

\begin{equation}\label{samebc}
\tilde y= \bar{\tilde y}.
\end{equation}
Now, it is well known that splitting sets are $c$-monotone \cite{KP2}, which implies:

\begin{eqnarray*}
c(x_1,x_2,x_3...,x_m) + c(x_1,\bar x_2,\bar x_3,...\bar x_m) &\leq& c(x_1,\bar x_2,x_3,...,x_m) \\
&&+ c(x_1, x_2,\bar x_3,...,\bar x_m).
\end{eqnarray*}
Using \eqref{samebc} and minimizing property of $\tilde y= \bar{\tilde y}$, this becomes
\begin{eqnarray*}
\sum_{i=1}^mc_i(x_i,\tilde y)+c_1(x_1,\tilde y)+\sum_{i=2}^m c_i(\bar x_i, \tilde y) &\leq &c (x_1,\bar x_2,x_3,...,x_m) \\
&& +c(x_1, x_2,\bar x_3,...,\bar x_m)\\
& \leq & \sum_{i \neq 2}^mc_i(x_i,\tilde y) +c_2(\bar x_2, \tilde y)\\
&& +\sum_{i = 3}^mc_i(\bar x_i, \tilde y) +c_1(x_1, \tilde y) +c_2 (x_2, \tilde y) 
\end{eqnarray*}
 But the first and last terms in the preceding string of inequalities are identical, and so we must have equality throughout.  In particular, we must have $c (x_1,\bar x_2,x_3,...,x_m) = \sum_{i \neq 2}^mc_i(x_i,\tilde y) +c_2(\bar x_2,\tilde y)$, so that $\tilde y \in argmin_y \sum_{i \neq 2}^m [c_i(x_i,y) + c_2(\bar x_2, y)]$.  Thus,

\begin{equation*}
\sum_{i \neq 2}^m D_yc_i(x_i,\tilde y) + D_yc_2(\bar x_2, \tilde y)=0,
\end{equation*}
or 
\begin{equation*}
 D_yc_2(\bar x_2, \tilde y)=-\sum_{i \neq 2}^m D_yc_i(x_i,\tilde y) = D_yc_2( x_2, \tilde y),
\end{equation*}
where the last equality follows from \eqref{firstordcond}.  Therefore, by the $(y,x_i)$-twist assumption, we have $\bar x_2 =x_2$ as desired.
\end{proof}

Let us note that there is significant overlap between the class of costs of form \eqref{infconcost} and functions satisfying the differential conditions from Proposition \ref{diffcond}; both classes include, for example, the concave functions of the sum  in subsection 2.3.2.  However, neither class contains the other; in \cite{P13a}, an example was exhibited that satisfies the differential conditions but is  not of  form \eqref{infconcost}, as well as one which is of form \eqref{infconcost}, but does not satisfy the differential conditions.  This was, in fact, part of the motivation behind the development of the twist on splitting sets condition; it is desirable to have a general condition encompassing all known examples.

However, we note that \textit{conditions}  on the $c_i$ (significantly stronger than those assumed in Proposition \ref{infcosttwist}) are known  under which  costs of form \eqref{infconcost} \textit{do} satisfy the differential conditions \cite{P1}.

\subsubsection{Vector fields costs}
We consider now the cost 
\begin{equation*}
c(x_1,x_2,x_3) = -(Ax_1 \cdot x_2 +Bx_1 \cdot x_3 + Ax_3 \cdot x_3 +Bx_3 \cdot x_2 +Ax_2 \cdot x_3 +Bx_2 \cdot x_1)
\end{equation*}
for $n \times n$ matrices $A,B$.  This cost shows up in a special case of a representation result in \cite{GhM}, for the linear vector fields $Ax_1, Bx_1$; the condition \textbf{(MMA1)} is relevant there as it implies uniqueness of the measure preserving $3$ - involution used to represent the vector fields.  The differential conditions for \textbf{(MMA1)} boil down to
\begin{equation*}
(A+B^T)[B+A^T]^{-1}(A+B^T)>0.
\end{equation*}
This holds if, for example, $M =B+A^T$ is symmetric positive definite, or if $M$ is unitary, $M^{-1}=M^T$, with $M^3 >0$  (for example, a rotation through an angle of less than $\frac{\pi}{6}$, or a rotation through an angle between $\pi/2$ and $5\pi/6$).
\subsection{Extensions and variants of the multi-marginal problem}
Recently, several variants of \eqref{mmk} have been introduced.  These include the optimal partial transport problem, where only a prescribed fraction of the mass is to be coupled, and the martingale optimal transport problem, where the minimization is restricted to martingale measures.  In the former variant, uniqueness of solutions for infimal convolution type costs has been established with Kitagawa \cite{KitPass}, under a natural and necessary condition on the marginals, extending work in the two marginal case of Caffarelli-McCann \cite{CM10} and Figalli \cite{Fig10a}.  In the later variant, for one dimensional marginals, explicit solutions for the two marginal case have been found \cite{bj} \cite{hpt} under natural conditions on the cost.  Here, the martingale constraint precludes the possibility of Monge solutions, but uniqueness persists.  The extension to several marginals has received a lot of attention due to applications in model independent bounds on derivative prices in mathematical finance.

Another extension is obtained when the number $m$ of marginals is allowed to be infinite.  In this case, the dichotomy described in the earlier examples becomes even more pronounced; for Gangbo-Swiech type costs, the solution again concentrates on graphs over the first marginal \cite{P5,P6}, but for a class of costs including the Coulomb cost (see subsection 3.2) the optimizer turns out to be product measure \cite{CFP}! 

In another variant, the minimization is restricted to measures with certain symmetry properties.  This has striking applications in the representation theory of vector fields \cite{GG}\cite{GhM}\cite{GhMa}, as well as potential applications in roommate type matching problems in economics \cite{cgs}.
\section{Applications}
In this section, we describe in detail two applications of multi-marginal optimal transport; hedonic matching for teams in economics, and density functional theory in physics.   In both cases, we first describe the application in some detail, and explain how multi-marginal optimal transport arises.  After this, we discuss what is known about the structure of the solution(s).  In a third subsection, we will briefly describe a variety of other applications.
%\subsection{Barycenters and geometry} 
\subsection{Economics: Matching for teams}
A model to study hedonic pricing was developed by Ekeland \cite{E}, and extended to multi-agent hedonic pricing, or matching for teams, by Carlier and Ekeland \cite{CE}.  They originally formulated the problem as a convex optimization problem in the space of measures; an equivalent formulation of the same problem as a minimization of type \eqref{mmk} can be found in both \cite{CE} and \cite{cmn}. 

As description is as follows.  Consider a population of buyers, looking to buy some good.   Production of that good requires input from several types of agents.  As a concrete example, suppose the population consists of buyers looking to purchase custom made houses.   To have a house built, a buyer must hire several subcontractors, say a carpenter, an electrician and  a plumber.  We will assume that the population of buyers is parameterized by a set $M_1$ and their relative frequency by a probability measure $\mu_1$ on $M_1$.   For $i=2,....m$ (where $m-1$ is the number of types of tradespeople), we will assume the population of the $i$th type of tradespeople is parameterized by a probability measure $\mu_i$ on a set $M_i$.  

We will assume that $Y$ represents a set of goods that could feasibly be constructed.  For example, $Y$ could be a subset of $\mathbb{R}^n$, and the different components $y^i$ of $(y^1,y^2,...,y^m) \in Y$ could measure different characteristics of the house, such as location, house size, lot size, etc.    Now, each buyer type $x_1$  has a preference for a  house of type $y \in Y$, $p_1(x_1,y)$ (which can be interpreted as the amount he feels house $y$ is worth)) and each subcontractor has a cost, $c_i(x_i,y)$, representing what it would cost him, in, say, labour and supplies, to do his share of the work on a house of type $y$.  
%An equilibrium is a an $m$-tuple of functions $u_i: Y \mathbb{R}$, interpreted as the wages each worker earns such that $\sum_i{=1}^m$

%An equilibrium in this model is defined as  
As is demonstrated by Carlier and Ekeland, setting $c_1(x_1,y)=-p_1(x_1,y)$, an equilibrium in this model turns out to be equivalent to solving the following variational problem:
\begin{equation}\label{gbc}
\inf_{\nu \in P(Y)} \sum_{i=1}^m T_{c_i}(\mu_i,\nu),
\end{equation}
 where $P(Y)$ represents the set of probability measures on $Y$ and 
\begin{equation}\label{taskassign}
T_{c_i}(\mu_i,\nu) :=\inf_{\lambda_i \in \Pi(\mu_i,\nu)} \int_{M_i \times Y} c_i(x_i,y)d\lambda_i
\end{equation}
is the optimal cost in the two marginal problem between $\mu_i$ and $\nu$.  Carlier and Ekeland also showed that this is equivalent to the multi-marginal problem \eqref{mmk} with cost function given by
\begin{equation}\label{matchingcost}
c(x_1,...,x_m) = \inf_{y \in Y} \sum_{i=1}^m c_i(x_i,y).
\end{equation}
Intuitively, for a given  distribution $\nu$ of houses (or \textit{contracts}), \eqref{taskassign} finds the matching of the $i$th category of workers to these contracts minimizing the total cost.  The problem \eqref{gbc} then looks for a $\nu$ that minimizes the sum of these costs over all categories.  Workers are then assigned to teams based on the contract $y \in Y$ that the optimal coupling $\lambda_i$ in \eqref{taskassign} matches them to. 

On the other hand, for a potential team $(x_1,x_2,...,x_m)$,  the minimization in \eqref{matchingcost} identifies which feasible contract $y$ would minimize the total cost for that particular team.  Problem \eqref{mmk} then asks how to form teams to minimize the overall total cost, assuming that each team will sign the overall best feasible contract for its purposes.
Precisely, the equivalence between these problems is captured by the following result (Proposition 3 in \cite{CE}).
\newtheorem{equivalence}{Proposition}[subsection]
\begin{equivalence}\label{equivalence}
Assume the infimum in \eqref{matchingcost} is uniquely attained for each $(x_1,x_2,...,x_m)$, at some point $\bar y(x_1,x_2,...,x_m)$.  Then
\begin{enumerate}
\item The infimums \eqref{mmk} (with cost \eqref{matchingcost}) and \eqref{gbc} have the same value.
\item If $\gamma$ is a solution to \eqref{mmk}, then $\bar y_{\#} \gamma$ is a solution to \eqref{gbc}.
\item If $\nu$ solves \eqref{gbc}, then there is a solution $\gamma$ to \eqref{mmk} such that $\nu =\bar y_{\#} \gamma$.
\end{enumerate}
\end{equivalence}
\newtheorem{relaxassume}[equivalence]{Remark}
\begin{relaxassume}
When $Y$ is an open domain in a smooth manifold, the uniqueness assumption on $\bar y$ can be relaxed somewhat, under other structural conditions on the $c_i$; in this case, if each $c_i$ is twisted, it suffices to assume the existence of a minimizer.  In fact, in this setting, if $\gamma$ solves \eqref{mmk} and $\mu_1$ is absolutely continuous with respect to  local coordinates, uniqueness for $\gamma$ almost all $(x_1,...,x_m)$ actually follows as a \emph{consequence} \cite{KP2}.
\end{relaxassume}
\newtheorem{minesfactories}[equivalence]{Remark}
\begin{minesfactories}
Interestingly, problem \eqref{gbc} has another natural interpretation, in terms of the factory and mine description of the classical optimal transport problem.   Suppose that a company is building a good whose production requires several resources, say iron, aluminum and nickel.  Imagine that the company has not yet built its factories where the goods will be constructed, but for each $i$ there is a probability measure $\mu_i$ on a set $M_i$  representing a distribution of mines producing the $i$th resource.  
The cost of shipping the $i$th resources from a mine at point $x_i$ to a (potential) factory at point $y$ is given by $c_i(x_i,y)$.  Imagine that the set $Y$ represents a set of locations where factories could conceivably be built; if the company built a distribution of factories on $Y$ according to a measure $\nu$, its transport cost to send resource $i$ from mines to factories would be $T_{c_i}(\mu_i,\nu)$. The company's goal is then to decide where to build the factories (that is, choose $\nu$) , and decide which mine of each type should supply each factory, in such a way as to minimize the total transportation cost.   This corresponds precisely to problem \eqref{gbc}.
\end{minesfactories}

When each $M_i=Y=M$, where $M$ is a Riemannian manifold, and $c_i =t_id^2(x_i,y)$ is a positive multiple of the Riemannian distance squared, then $T_{c_i}(\nu,\mu_i) =t_iW^2(\nu,\mu_i)$ is proportional the squared Wasserstein distance, and solutions to \eqref{gbc} are \textit{barycenters} of the measures $\mu_i$ with weights $t_i$, with respect to the Wasserstein metric on $P(M)$.    This was investigated by Agueh-Carlier when $M= \mathbb{R}^n$, and they were able to use the equivalence with the multi-marginal problem (with Gangbo-Swiech cost in this case)  to establish a regularity result on the barycenter \cite{AC} (the present author later extended this technique to more general costs $c_i(x_i,y)$  on $\mathbb{R}^n$, establishing absolute continuity of the solution $\nu$ to \eqref{gbc}\cite{P13a}).   The  Riemannian case was studied with Kim \cite{KP}, and the uniqueness and Monge solution results obtained there can be viewed as an extension of the Gangbo-Swiech theorem \cite{GS} to curved geometries, analogous to McCann's extension \cite{m3} of Brenier's polar factorization theorem \cite{bren}.  Barycenters are interesting in their own right, and also have applications arising in texture mixing in image processing \cite{bdpr} as well as statistics \cite{BK} . 

The twist condition on $c_1$ and regularity of the first marginal is sufficient to ensure uniqueness of the solution to \eqref{gbc} \cite{CE}.  Turning to the multi-marginal problem, provided the minimum in \eqref{matchingcost} is attained for all $(x_1,...,x_m)$, then \eqref{matchingcost} is identical to \eqref{infconcost}, and so Theorem \ref{mmmapping} and Proposition \ref{infcosttwist} provide conditions ensuring uniqueness and Monge structure of the solution to \eqref{mmk}.

%\newtheorem{hedonic}[equivalence]{Theorem}
%\begin{hedonic}
%Assume $c_1$ is $(x_1,z)$-twisted and $c_i$ is $(z,x_i)$-twisted for $i=2,3,...m$.  Then the cost function \eqref{matchingcost} is twisted on splitting sets.  
%\end{hedonic} 

%\newtheorem{hedonicmonge}[equivalence]{Corollary}
%\begin{hedonicmonge}
%Assume $c_1$ is $(x_1,z)$-twisted and $c_i$ is $(z,x_i)$-twisted for $i=2,3,...m$.  Then the multi-marginal problem with cost function \eqref{matchingcost} %admits a unique solution and the solution is of Monge type.
%\end{hedonicmonge}

In this context, the Monge solution structure is equivalent to a version of the economic notion of \emph{purity}; it means that, in equilibrium, buyers of the same type $x_1$ almost surely match with workers $x_i =F_i(x_1)$ of the same type in category $i$, for each $i$.  When $m=2$, this notion of purity was introduced in this context by Chiappori, McCann and Nesheim \cite{cmn}.  A different notion of purity, essentially that workers of the same type $x_1$ always buy goods of the same type $y \in Y$ was discussed by Carlier and Ekeland, and requires only twistedness of the first cost $c_1$ \cite{CE}.
\subsection{Physics: Density Functional Theory}
A natural and fundamental problem in quantum physics is to determine or estimate the ground state energy of a system of interacting electrons.  Quantum mechanically, such a system in modeled by an $m$-particle wave function.  Neglecting the role of spin, this amounts to a complex valued, square integrable function $\psi(x_1,....,x_m)$ of the respective positions $x_1,...,x_m \in \mathbb{R}^n$ of the electrons, with $n=3$ being the most physically relevant case.  The wave function is required to be anti-symmetric: for any permutation $\sigma$ on $m$ letters, we have $\psi(x_1,x_2,...,x_m) = sgn(\sigma)\psi(x_{\sigma(1)},x_{\sigma(2)},...,x_{\sigma(m)})$.  

Physically, $|\psi(x_1,....,x_m)|^2$ represents the probability that the electrons will be found at positions $(x_1,...,x_m)$ respectively; note that the anti-symmetry of $\psi$ implies the symmetry of this probability; we can interpret this as expressing  indistinguishability of the electrons.

The total energy of the system is given by 

\begin{equation}
E[\psi] = T[\psi] + V_{ext}[\psi] + V_{ee}[\psi], 
\end{equation}
where $T[\psi] = \frac{\hbar}{2m_e}\int_{\mathbb{R}^{nm}}\sum_{i=1}^m|\nabla_{x_i} \psi|^2|\psi|^2dx_1...dx_m$ is the kinetic energy ($\hbar$ is Planck's constant and $m_e$ the mass of an electron), $V_{ext}[\psi] = m\int_{\mathbb{R}^n}V(x)d\mu(x)$ is the external energy arising from a potential function $V$ ($\mu$ is the single particle density of the wavefunction $\psi$,  or equivalently, the marginal on each copy of $\mathbb{R}^n$ of the symmetric measure on $\mathbb{R}^{nm}$ with density $|\psi(x_1,x_2,...,x_m)|^2$) and  $V_{ee}[\psi] = {m \choose 2}\int_{\mathbb{R}^{nm}} \sum_{i \neq j}^m \frac{1}{|x_i-x_j| }|\psi|^2dx_1...dx_m$ is the electronic interaction energy.  We are interested in minimizing $E$ over all wave functions $\psi$.  In this form, the problem is numerically unwieldy, as we must minimize over all wave functions $\psi$ on $\mathbb{R}^{nm}$, and $m$ can be quite large for many systems.  The problem can, however, be rewritten as an iterated minimization:

\begin{equation*}
\min_{\mu} V_{ext}[\mu] +F_{HK}[\mu].
\end{equation*}
where $F_{HK}[\mu] := \min_{ \psi \rightarrow \mu}T[\psi]+V_{ee}[\psi]$ is the Hohenberg-Kohn functional, and the notation $\psi \rightarrow \mu$ means that $\mu$ is the single particle density of the wavefunction $\psi$.    This represents a substantial complexity reduction, as we are minimizing over single particle densities $\mu$ rather than $m$-particle wave functions $\psi$, \textit{if} the functional $F_{HK}[\mu]$ is well understood (note that $F_{HK}[\mu]$ contains in its definition a minimization over wave functions $\psi$, and so the complexity is reduced only if we have an analytic way to determine, or at least approximate, $F_{HK}[\mu]$).  Density functional theory is the study of this functional and is a hugely popular area of research among physicists and chemists.  The main goal is to find good approximations of $F_{HK}[\mu]$; for a detailed introduction, see \cite{PY}.

Although it had in some sense been implicit in the physics literature for quite some time \cite{sggs} \cite{seidl}, the formulation of density functional theory as a  multi-marginal optimal transport problem was only recently introduced, independently by Cotar, Friesecke and Kluppelberg \cite{CFK} and Buttazzo, De Pascale and Gori-Giorgi \cite{bdpgg}.  In particular, the precise link  comes when we consider the semi-classical limit, $\hbar \rightarrow 0$, as is shown by the following theorem of Cotar, Friesecke and Kluppelberg \cite{CFK} \cite{cfk2}.

\newtheorem{semiclasslim}{Theorem}[subsection]

\begin{semiclasslim}
In the limit as $\hbar$ tends to $0$, $F_{HK}[\mu]$ is equal to the infimal value in \eqref{mmk}  with the Coulomb cost $c(x_1,x_2,...,x_m):=\sum_{i \neq j}^m \frac{1}{|x_i-x_j| }$ and each marginal $\mu_i =\mu$; that is:
$$
\lim_{\hbar \rightarrow 0} F_{HK}[\mu] ={m \choose 2}\inf_{\mu \in \Pi(\mu,\mu,...,\mu)}C(\gamma),
$$
where $C(\gamma):=\int_{\mathbb{R}^{nm}}\sum_{i \neq j}^m \frac{1}{|x_i-x_j| }d\gamma$.
\end{semiclasslim}

\newtheorem{symm}[semiclasslim]{Remark}

\begin{symm}
The minimization above  may be restricted to  measures which are symmetric under permutations; that is, measures for which $\gamma =\sigma_{\#}\gamma$ for any permutation $\sigma$ of the arguments $(x_1,x_2,...,x_m)$.  This follow easily by noting that $C(\gamma) = C(\bar \gamma)$, where  $\bar \gamma =
\frac{1}{m!}\sum_{\sigma \in S_m}\sigma_{\#} \gamma$ is the symmetrization of $\gamma$ (the sum is over the permutation group $S_m$ on $m$ letters), and that if $\gamma \in \Pi(\mu,...,\mu)$ then so does $\bar \gamma$.

Physically, we can think of symmetric measures as representing semi-classical $m$-particle densities, with symmetry reflecting the indistinguishable of the electrons.  Mathematically, the observation above raises subtle uniqueness questions, as it is possible for the solution $\gamma$ to be non-unique, but for there to be a unique minimizer in the smaller class of symmetric minimizers (see Theorem \ref{dft1d} below).

The study of optimal transport problems with symmetry constraints was initiated (with a different cost function, and very different motivations) in the two marginal setting by Ghoussoub and Moameni \cite{GhM2}, and continued in the multi-marginal context  in a recent series of papers \cite{GhM}\cite{GG}\cite{GhMa}.
\end{symm}

The following theorem from \cite{CFK} and \cite{bdpgg} is similar in spirit to Theorem \ref{mapping}, as  the two marginal Coulomb cost satisfies the twist condition off the diagonal, although the proof must overcome some additional technical obstacles due to the singularity on the diagonal.

\newtheorem{dft2marg}[semiclasslim]{Theorem}

\begin{dft2marg}
Let $m=2$, assume that $\mu$ is absolutely continuous with a density in $L^1 \cap L^3$.  Then the optimal measure $\gamma$ induces a Monge solution and is unique.
\end{dft2marg}

Note that the uniqueness immediately implies symmetry under permutations of the optimizer $\gamma$.  When the marginal $\mu$ is radially symmetric, one can use Theorem \ref{radial} to deduce an explicit formula for the optimal map, arriving at $F(x) = -xf(|x|)$, where $f$ is the monotone decreasing rearrangement of the radial density function.  This was originally proven in \cite{CFK} and \cite{bdpgg}; the proof based on the Theorem \ref{radial} can be found in \cite{P13b}.
 
As we will see below, the graphical structure and uniqueness break down as soon as there are at least three marginals.  In one dimension, however, there is still a unique \textit{symmetric} minimizer, as the following result of Colombo, De Pascale and Di Marino  shows \cite{cdpdm}.

\newtheorem{dft1d}[semiclasslim]{Theorem}
\begin{dft1d}\label{dft1d}
Let $\mu$ be a  non  atomic  probability measure on $\mathbb{R}$. Choose points $-\infty =d_0 <d_1...<d_m = \infty$ such that $\mu[d_i,d_{i+1}] =\frac{1}{m}$.  Then define $T:\mathbb{R} \rightarrow \mathbb{R}$ by letting its restriction to the interval $[d_i,d_{i+1}]$ be the unique monotone increasing map pushing $\mu|_{[d_i,d_{i+1}]}$ forward to $\mu|_{[d_{i+1},d_{i+2}]}$ for $i=0,...m-1$, using the convention $m+1 =0$.

Then $(T,T^2,T^3,...T^{m-1})$ is a solution to \eqref{mmm} with Coulomb cost, where $T^k$ denotes composition of $T$ with itself $k$ times.  Furthermore, the symmetrization of $\gamma=(T,T^2,T^3,...T^{m-1})_{\#}\mu$ is the unique symmetric solution to \eqref{mmk}
\end{dft1d}
\newtheorem{localmono}[semiclasslim]{Remark}

\begin{localmono}
On any open set where $c$ is non-singular, (for instance, the set where $x_1 <x_2<...<x_m$) we have $\frac{\partial^2 c}{\partial x_i \partial x_j} <0$, and so $c$ is compatible when restricted to this set, and the local structure of  the solution in the preceding theorem is consistent with Theorem \ref{onedmonge}.  But clearly the cost is not (globally) twisted on splitting sets, as solutions can be found which are not concentrated on graphs.
\end{localmono}
In contrast to the $n=1$ case, in higher dimensions Corollary \ref{highdnonattractive} immediately implies the solution may concentrate on sets with dimension at least $2n-2$ ($=4$ in the physically relevant $m=3$ case).  Uniqueness of the solution fails and in fact, under additional, somewhat technical hypotheses, we can also conclude that there is more than one \emph{symmetric} minimizer, a distinction from the  $1$-dimensional case \cite{P13b}.

\newtheorem{quaddft}[semiclasslim]{Remark}
\begin{quaddft}
Consider replacing the Coulomb cost with the repulsive harmonic oscillator cost $-\sum_{i\neq j} |x_i-x_j|^2$, representing a milder, but still repulsive interaction.  Though clearly not as physically relevant as the Coulomb cost, this interaction has received considerable attention in the physics literature, partially because it is easier to deal with analytically \cite{bdpgg}\cite{sggs}.
 
Noting that this cost is equivalent to the convex function $|\sum_{i=1}^mx_i|^2$ of the sum,  Remark \ref{highdconvex} implies that the solution can be concentrated on manifolds of dimension $(m-1)n$; see \cite{P} and \cite{FMPCK} for details.

For the Coulomb cost, we strongly suspect that solutions cannot concentrate on $(m-1)n$-dimensional surfaces.  It remains an open question to estimate the dimension more precisely, and in particular, determine whether $2n-2$ is the worst possible case.
\end{quaddft} 

Monge solutions are often called \textit{strictly correlated electrons} in the physics literature.  The natural physical interpretation is that  the position of the first electron determines the positions of the others.  The above discussion indicates that there are certain examples when the electrons are not strictly correlated; we note here that we can \textit{never} have \textit{symmetric}  Monge solutions when $m \geq 3$.  We illustrate the reason for this in the $m=3$ case; symmetry under the permutation $(x_1,x_2,x_3) \rightarrow (x_1,x_3,x_2) $ of the measure concentrated on the graph $\{(x,T_2(x),T_3(x))\}$ easily implies $T_2 =T_3:=T$ almost everywhere.  Symmetry under the permutation  $(x_1,x_2,x_3) \rightarrow (x_2,x_1,x_3)$ then yields that points of the form $(T(x),x,T(x))$ are in the support of $\gamma$, which means that $T(T(x)) =x$  and $T(T(x)=T(x)$ .  This means that $T(x) =x$ almost everywhere; the only graphical symmetric measure is induced by the identity mapping $\gamma=(Id,Id,...,Id)_{\#}\mu$, which yields infinite energy, $C(\gamma) =\infty$, and cannot be optimal.  This immediately implies that Monge solutions and uniqueness cannot coexist, as uniqueness immediately implies permutation symmetry.  All of this holds for any other permutation symmetric cost when the marginals are equal (unless $\gamma=(Id,Id,...,Id)_{\#}\mu$ is in fact optimal, as is the case with, for example, the Gangbo-Swiech cost).

On the other hand, it remains an interesting open question to determine whether there \emph{exists} a Monge type solution to \eqref{mmk} for the Coulomb cost when $m \geq 3$ and $n \geq 2$.  In this direction,  one can at least approximate optimal measures by measures concentrated on graphs $\{(x_1,F_2(x_1),...,F_m(x_1))\}$, and one can even take $F_i =T^{i-1}$, for some measuring preserving $T$ with $T^{m+1} =Id$ \cite{cdm}.  These approximating measures are then symmetric under the cyclic permutation  $(x_1,x_2,...,x_3) \rightarrow (x_m,x_2,...,x_{m-1}) $, but not under more general permutations. 

In another direction, one can notice that as the Coulomb cost is a sum of two body interactions, one can rewrite \eqref{mmk} (with a symmetry constraint) as a minimization over measures $\mu_2$ on $\mathbb{R}^{2n}$, with an additional \emph{representability} restriction (essentially that $\mu_2$ must be the two body marginal of some measure $\gamma$ on $\mathbb{R}^{nm}$) \cite{FMPCK}.  This approach is useful in considering the limit $m \rightarrow \infty$, where it turns out that product measure (known as the mean-field functional in physics) is optimal \cite{CFP}. 
\subsection{Other applications}
While our focus here has been on applications in economics and physics, we would be remiss if we neglected to at least briefly mention a wide variety of other interesting applications for multi-marginal problems. These  include texture mixing in image processing, where the aim is to interpolate among several textures (encoded as measures) in order to synthesize a new one \cite{bdpr}, and statistics, where the basic goal is to find an average among several distributions while preserving relevant features of them \cite{BK}.  These first two areas lead to costs of the same form as in the economic matching problem \eqref{gbc} -- and in fact the barycenter of the measures is actually the fundamental object of interest here.

Other applications arise in financial mathematics. Recent work in this direction has actually focused on a variant of \eqref{mmk}, where we restrict the minimization to the set of martingale measure.  This problem arises in the derivation of model independent price bounds.  Here, the measures $\mu_i$ represent (risk neutral) distributions of the price of an asset at a sequence of future times.  Given this information, one wishes to determine the price of an exotic derivative, whose payoff $c(x_1,x_2,...,x_m)$ depends on the values $x_i$ of the underling asset at each of the times $t_i$; one can use a variety of models in mathematical finance to derive measures $\gamma$, representing a coupling of the distributions $\mu_i$, and the integral in \eqref{mmk} then represents price of the derivative. This of course depends on the measure $\gamma$, which in turn depends on the financial  model used to derive it.  Arbitrage free models always yield discrete time martingale measures, and so the minimization \eqref{mmk} with the additional constraint that $\gamma$ should be a martingale measure, tells us the minimal arbitrage free price of the derivative, among all possible models.  This has quickly become a very active area; see for example \cite{ds3}\cite{ds2}\cite{ght}\cite{hpt}\cite{BeigGriess}\cite{bhlp}\cite{ght}\cite{hltt}.

In a related financial application, a derivative's payoff $c$ may depend on the prices $x_i$ of several different assets at the same time, whose distributions $\mu_i$ are known.  Here, to find the minimal possible value, we minimize $\int c(x_1,...x_m)d\gamma(x_1,...,x_m)$ over $\Pi(\mu_1,...,\mu_m)$.  This is exactly \eqref{mmk}; there is no reason in this case to restrict to the smaller class of martingale measures; see, for example, \cite{pr}\cite{eprwb}.

We  also mention that problem \eqref{mmk} arises in the time discretization of the least action principle for incompressible fluids; see, for example, \cite{bren2}.  The cost function arising there takes the form $\sum_{i=1}^{m-1}|x_i-x_{i+1}|^2 +|x_{m} -F(x_1)|^2$, where $F$ is a prescribed measure preserving diffeomorphism. To the best of our knowledge, very little is known about the structure of solutions with this type of cost.

Finally, let us mention some applications of \eqref{mmk} in pure mathematics.,  These include a recent series of representation theorems for families of vector fields \cite{GG}\cite{GhM}\cite{GhMa}; for example, given any $(m-1)$-tuple $(u_2(x),...,u_m(x))$ of measurable vector fields on $\Omega \subset \mathbb{R}^n$, a symmetric variant of $\eqref{mmk}$ can be exploited to establish the existence of an $m$-cyclically anti-symmetric, concave-convex Hamiltonian $H$ and a measure preserving $m$-involution $S$ such that $(u_2(x),...,u_m(x)) \in \partial_{x_2,....,x_m}H(x_1,S(x),S^2(x),...,S^{m-1}(x))$.  In addition, \eqref{mmk} has proven useful in resolving variational problems on Wasserstein space \cite{Dahl}, and in the decoupling of systems of elliptic PDEs \cite{GhP1}.

\subsection{Numerics}
Numerics for multi-marginal problems have so far not been extensively developed (and even for two marginal problems, where numerical schemes have received more attention, there remain many important and challenging open issues).  Discretizing the multi-marginal problem leads to a linear program where the number of constraints grows exponentially in $m$, the number of marginals.  Very recently, however, a paper of Carlier, Oberman and Oudet studied the matching for teams  problem \eqref{gbc} \cite{coo}.  Among their findings, they were able to reformulate the problem as a linear program  whose number of constraints grows only linearly in $m$, which is much more amenable to numerics than naive linear programming techniques for the multi-marginal problem \eqref{mmk}.  Using Proposition \ref{equivalence}, these techniques yields equally tractable numerical schemes for \eqref{mmk} when the cost is of the form \eqref{infconcost}.  Performance is improved even further when each $c_i$ is quadratic, as special features of the cost can then be used to solve the barycenter problem \eqref{gbc}.   It is interesting to note that early numerical success reflects the dichotomy described in the introduction, and further exposed in the rest of this manuscript; matching type costs have low dimensional solutions, and so can be expected to be reasonably well behaved from a numerical point of view, whereas costs such as the Coulomb cost, or the repulsive harmonic oscillator cost, can yield high dimensional solutions, and so developing widely applicable numerical algorithms for these costs is likely to pose additional challenges.

Let us mention, however, that the Coulomb cost is tackled in \cite{ML}, and in several examples quite accurate solutions can be computed using parameterized functional forms of the Kantorovich potential. In the two marginal case, a linear programming approach for the Coulomb cost is implemented in \cite{cfm}. In \cite{bdpr}, Barycenter problems are addressed by replacing the  Wasserstein distance with the sliced Wasserstein distance, which is much easier to compute than the true Wasserstein distance.
\bibliographystyle{plain}
\bibliography{biblio}
\end{document}